\documentclass[5pt]{article}
\usepackage{amsmath}
\usepackage[latin1]{inputenc}
\usepackage{amsmath,amsthm,amssymb}
\usepackage{graphicx}
\usepackage[all,poly,knot]{xy}

\usepackage{color}

\renewcommand{\theequation}{\thesection.\arabic{equation}}

\newtheorem{defn}{Definition}[section]
\newtheorem{lem}{Lemma}[section]
\newtheorem{thm}{Theorem} [section]

\newtheorem{coro}{Corollary}[section]

\title{ Permutation Polynomials and their Compositional Inverses
\thanks{Supported By NSF of China No. 12171163 }}

\author{Pingzhi Yuan\thanks{ P. Yuan is with School of  of Mathematical Science, South China Normal University,  Guangzhou 510631, China (email: yuanpz@scnu.edu.cn).}}

    \date{}
\begin{document}
\baselineskip15pt \maketitle
\renewcommand{\theequation}{\arabic{section}.\arabic{equation}}
\catcode`@=11 \@addtoreset{equation}{section} \catcode`@=12

    \begin{abstract}In this paper, we  prove that every PP is an AGW-PP. We also extend the result of Wan and Lidl  to  other permutation polynomials over finite fields and determine their group structure. Moreover, we provide a new  general method to find the compositional inverses of all PPs, some new PPs and their compositional inverses are given.

\end{abstract}


\section{Introduction}

\,\,\, Let $q$ be a prime power, $\mathbb{F}_q$ be the finite field of order $q$, and $\mathbb{F}_q[x]$
be the ring of polynomials in a single indeterminate $x$ over $\mathbb{F}_q$. A polynomial
$f \in\mathbb{F}_q[x]$ is called a {\em permutation polynomial} (PP for short) of $\mathbb{F}_q$ if it induces
a bijective map from $\mathbb{F}_q$ to itself. The  unique polynomial denoted by $f^{-1}(x)$ over $\mathbb{F}_q$ such that $f\circ f^{-1}=f^{-1}\circ f=I$ is called the compositional inverse of $f(x)$, where $I$ is the identity map.  Furthermore, we call $f (x)$ an
involution when $f^{-1}(x)=f(x)$.

Permutation polynomials over finite fields  and their compositional inverses have wide applications in coding
theory \cite{Ding13, DZ14, LC07}, cryptography \cite{RSA, SH}, combinatorial
design theory \cite{DY06}, and other
areas of mathematics and engineering \cite{LN97,LN86, Mull}.

Constructing PPs and explicitly determining the compositional inverse of a PP are useful because a PP and its inverse are required in many applications. For example, involutions are particularly useful ( as a part of a block ciper) in devices with limited resources. However, it is difficult to determine whether  a given polynomial over a finite field is a PP or not, and it is not easy to find the explicit compositional inverse of a random PP, except for several well-known classes of PPs, which have very nice structure. See  see \cite{CH02, LQW19, LN97, NLQW21, TW14, TW17, W17, Wu14, WL13, WL13J, ZYLHZ19, ZY18, ZWW20} for more details.

In 2011, Akbrary, Ghioca and Wang  \cite{AGW11} proposed a powerful method called the AGW criterion for constructing PPs. A PP is called AGW-PP when a PP is constructed using the AGW criterion or it can be interpreted by the  AGW criterion \cite{NLQW21}.    Niu, Li, Qu and Wang \cite{NLQW21} obtained a general method to finding compositional inverses of AGW-PPs. They obtained the compositional inverses of all AGW-PPs of the type $x^rh(x^s)$ over $\mathbb{F}_q$, where $s|q-1$.  Recently, the author \cite{Y22} used the dual diagram to improve the results in \cite{NLQW21}.

Permutation polynomials of the form $x^rh(x^{(q-1)/d})$ have been paid particular attention, where $d$ is a divisor of $q-1$. These PPs originated from
the work of Rogers and Dickson [2] in the 19th century, who first treated the case $h(x)=g(x)^d$. For the group structure of these PPs, Wells ([17], [18]) proved that in the case $r = 1$ the set of PPs of the form $xf(x^{ (q-1)/d})$ is a group isomorphic to the wreath product of a cyclic group of order $(q-1)/d$ with the
symmetric group $S_d$. Wan and Lidl  \cite{WL17} showed that all such PPs form a group isomorphic to a generalized wreath product of certain abelian groups.

In the present paper, we first prove that every PP is an AGW-PP. Next we extend the result of Wan and Lidl \cite{WL17} to  other permutation polynomials over finite fields and determine their group structure. Finally, we present a general method to find the compositional inverses of PPs.

The rest of this paper is organized as follows. In Section 2, we present some preliminaries which are used in the sequel. In Section 3, we prove that every PP is an AGW-PP. In Section 4, we use the  diagram in Section 2 to determine their group structure of PPs with some special properties. We present a general method to find the compositional inverses of PPs in Section 5. As an application the results in Section 5, we present some new PPs and their compositional inverses in Section 6.

\section{Preliminaries}

The following lemma is taken from \cite[Lemma 1.1]{AGW11}, which is called AGW criterion now.

\begin{lem}\label{lem-1.1}   {\rm ( AGW criterion) }
Let $A, S$ and $\bar{S}$ be finite sets with $\sharp
S=\sharp\bar{S}$, and let $f:A\rightarrow A, h:
S\rightarrow\bar{S}$, $\lambda: A\rightarrow S$, and $\bar{\lambda}:
A\rightarrow \bar{S}$ be maps such that $\bar{\lambda}\circ
f=h\circ\lambda$. If both $\lambda$ and $\bar{\lambda}$ are
surjective, then the following statements are equivalent:

(i) $f$ is a bijective (a permutation of $A$); and

(ii) $h$ is a bijective from $S$ to $\bar{S}$ and if $f$ is
injective on $\lambda^{-1}(s)$ for each $s\in S$.
\end{lem}

The following definitions can be found in \cite[Chapter 7]{Rot94}.

Let $K$ be a (not necessarily normal) subgroup of a group $G$. Then a subgroup $Q\le G$ is a {\it complement} of $K$ in $G$ if $K\cap Q=1$ and $KQ=G$, where $KQ=\{kq, \, k\in K, q\in Q\}$.

\begin{defn} A group $G$ is a {\bf semidirect product} of $K$ by $Q$, denoted by $G=K\rtimes Q$, if $K\triangleleft G$ and $K$ has a complement $Q_1\cong Q$. One also says that $G$ {\bf splits} over $K$. \end{defn}

\begin{defn}Let $D$ and $Q$ be groups, let $\Omega$ be a finite $Q$-set, and let $K=\prod_{\omega\in\Omega}D_\omega$, where $D_\omega\cong D$ for all $\omega\in\Omega$. Then the {\bf wreath product} of $D$ by $Q$ denoted by $D\wr Q$ is the semiproduct of $K$ by $Q$, where $Q$ acts on $K$ by $q\cdot(d_\omega)=(d_{q\omega})$ for $q\in Q$ and $d_\omega)\in \prod_{\omega\in\Omega}D_\omega$. The normal subgroup $K$ of $D\wr Q$ is called the {\bf base} of the wreath product. \end{defn}

If $D$ is finite, then $|K|=|D|^{|\Omega|}$, if $Q$ is also finite, then $|D\wr Q|=|K\triangleleft Q|=|K||Q|=|D|^{|\Omega|}|Q|$.

\section{AGW-PP or not?}
A PP is called AGW-PP when a PP is constructed using the AGW criterion or it can be interpreted by the  AGW criterion in \cite{NLQW21}. In this  section, we first prove that  every PP is  an AGW-PP, so we will not use the notion AGW-PP from now on. We have
\begin{thm}\label{AGW}Let $A, S$ be finite sets and let $f:A\rightarrow A$ be a map and
 $\varphi: A\rightarrow S$ be a surjective map. Then $f$ is a bijection if and only if the following conditions hold:

 (i) $f(x)$ is injective on each $\lambda^{-1}(s)$ for all $s\in S$;

 (ii) for each bijection $h: S\rightarrow S$, there exists an uniquely determined surjective map $\psi: A\rightarrow S$ such that the following diagram commutes
 $$\xymatrix{
  A \ar[d]_{\varphi} \ar[r]^{f}
                & A \ar[d]^{\psi}  \\
  S  \ar[r]_{h}
                & S          }$$

 \end{thm}
\begin{proof} Suppose that $f$ is a bijection, then (i) holds trivially. For each bijection $h: S\rightarrow S, s\mapsto h(s), s\in S$, we define the map
$$\psi: A \rightarrow S, \quad \psi(a)=h(s) $$
for all $a\in f(\varphi^{-1}(s)), s\in S$, where $\varphi^{-1}(s)$ is the inverse image of $s$ in $A$ and $f(\varphi^{-1}(s))=\{f(a), a\in \varphi^{-1}(s)\}$ is the set of all images of $a\in \varphi^{-1}(s)$.
Then
$$\psi\circ  f(a)=\psi(f(a))=h(s)=h(\varphi(a))=h\circ \varphi(a)$$
for any $a\in\varphi^{-1}(s)$, and it follows that $\psi\circ f=h\circ \varphi$, i.e. the above diagram commutes. It is easy to see that the surjective $\psi: A \rightarrow S$  is uniquely determined by $f, \varphi, $ and $h$.

The other direction follows directly from AGW criterion. \end{proof}

{\bf Remark:} By Theorem \ref{AGW}, we see that for any given bijection $f:A\rightarrow A$, a surjective
 $\varphi: A\rightarrow S$ and a  bijection $h: S\rightarrow S$, there exists an uniquely determined surjective map $\psi: A\rightarrow S$ such that the following diagram commutes
 $$\xymatrix{
  A \ar[d]_{\varphi} \ar[r]^{f}
                & A \ar[d]^{\psi}  \\
  S  \ar[r]_{h}
                & S          }$$

\section{The Group Structure of PPs}
In 1991,  Wan and Lidl \cite{WL17} showed that all permutation polynomials of the form $x^rh(x^s)$ from a group $G(d, q)$ under composition and this group is isomorphic to a generalized wreath product. In this section, we will prove a general result. We need the following two lemmas.

\begin{lem}{\rm \cite[Theorem 4.1]{Y22}} \label{Gp} Let $A, S$ be finite sets and let
 $\varphi: A\rightarrow S$ be a surjective map. Let $G_\varphi$ be the set of all bijections $f: A\rightarrow A$ such that $\pi_f\circ \varphi=\varphi\circ f$, where $\pi_f:S\rightarrow S$ is a bijection. Then $G_\varphi$ is a group under composition.\end{lem}

 \begin{lem}{\rm \cite[Theorem 4.2]{Y22}}\label{Gi} Let the notations be as in Lemma \ref{Gp}, and let
$$G(\varphi, 1)=\{f\in G_\varphi, \pi_f=1\}.$$Then $G(\varphi, 1)$ is a normal subgroup of $G_\varphi$.\end{lem}
%


If $A$ is a group and $\varphi$ is an epimorphism, then we have the following theorem.

\begin{thm}Let $A$ be a finite group of order $n$. Let $S\subseteq A$ be a subgroup with $|S|=d$ and $\varphi: A\rightarrow S$ be an epimorphism.  Let $G_\varphi$ be the set of all bijections $f: A\rightarrow A$ such that $\pi_f\circ \varphi=\varphi\circ f$, where $\pi_f:S\rightarrow S$ is a bijection. Then $G_\varphi=S_{n/d}\wr S_d$ and
$$\sharp G_\varphi=\left(\left(\frac{n}{d}\right)!\right)^dd!.$$\end{thm}
\begin{proof} Since $A$ is a finite group and $\varphi: A\rightarrow S$ is an epimorphism, we denote
$$S=\{s_1, \ldots, s_d\}, \quad \pi_f: s_i\to s_{\pi_f(i)},$$
then $\{\pi_f(1), \ldots, \pi_f(d)\}$ is a permutation of $\{1, 2, \ldots, d\}$.
Now $|\varphi^{-1}(s_i)|=n/d$, $i=1, 2, \ldots, d$ is a constant for all $i$ as $\varphi$ is an epimorphism. Put
$$\varphi^{-1}(s_i)=C_i,\,\,\, i=1, 2, \ldots, d.$$
Since $\pi_f\circ\varphi=\varphi\circ f$, it follows that
$$f(C_i)=C_{\pi_f(i)}, \,\,\, i=1, 2, \ldots, d,$$
and $f|_{C_i}: C_i\rightarrow C_{\pi_f(i)}$ is a bijection from $C_i$ to $C_{\pi_f(i)}$, and we denote this local map by $f_i, i=1, 2, \ldots, d$. Now we can write the bijection $f$ as
$$f=(f_i)_{i=1}^d=(f_i) \, \mbox{ for short},$$
and
$$G_\varphi=\{(f, \pi_f)=((f_i), \pi_f)\}.$$
Then the product of two elements in $G_\varphi$ can be expressed as
$$((g_i), \pi_g))\circ((f_i), \pi_f))=((g_{\pi_f(i)}\circ f_i), \pi_g\circ\pi_f),$$
where $g_{\pi_f(i)}: C_{\pi_f(i)}\rightarrow C_{\pi_{(g\circ f)}(i)}$ and $\pi_{(g\circ f)}(i)=\pi_g(\pi_f(i))$.

Let $D=S_{n/d}$ and $D=S_d$ be symmetric groups. $\Omega=\{1, 2,\ldots, d\}$,
$$K=\prod_{\omega=1}^dD_\omega, \quad D_\omega\cong D_{n/d}, \omega\in\{1, 2,\ldots, d\}.$$
For any $\pi\in S_d$ and $(f_\omega)\in \prod_{\omega=1}^dD_\omega$,
$$\pi((f_\omega))=(f_{\pi(\omega)}).$$
By the definition of wreath product, we have that $G_\varphi$ is the wreath product of $S_{n/d}$ and $S_d$, and
$$\sharp G_\varphi=\left(\left(\frac{n}{d}\right)!\right)^dd!.$$ The base group is
$$K=\{((f_\omega),1), (f_\omega)\in \prod_{\omega=1}^dD_\omega\},$$
which is a normal subgroup of $G_\varphi$ by Lemma \ref{Gi}. This proves the theorem.
\end{proof}

\section{Compositional Inverses of PPs}
We first prove the following general theorem.

\begin{thm} \label{Gv}Let $s, t$ be  positive integers. Let $A, S_i, T_j, i=1, \ldots, s,  j=1, \ldots, t$ be non-empty finite sets, and let $\varphi_i: A \rightarrow S_i, 1\le i\le s$, $\psi_j: A \rightarrow T_j, 1\le j\le t$, $h: S_1\times\cdots\times S_s\rightarrow T_1\times\cdots\times T_t$ be maps such that the maps $ \varphi: A\rightarrow S_1\times\cdots\times S_s$, $\varphi(a)=(\varphi_1(a), \ldots, \varphi_s(a)), a\in A$ and $\psi: A\rightarrow T_1\times\cdots\times T_t$, $\psi(a)=(\psi_1(a), \ldots, \psi_t(a)), a\in A$ are injective and the following diagram commutes
 $$\xymatrix{
  A \ar[d]_{(\varphi_1, \ldots, \varphi_s)} \ar[r]^{f}
                & A \ar[d]^{(\psi_1, \ldots, \psi_t)}  \\
  S_1\times\cdots\times S_s  \ar[r]_{h}
                &   T_1\times\cdots\times T_t        }$$
Then $f$ is a bijection if and only if $h|_{\varphi(A)}: \varphi(A)\rightarrow T_1\times\cdots\times T_t$ is an injective. Further, if $f$ is a bijection,  $\varphi^{-1}:S_1\times\cdots\times S_s\rightarrow A$ and $(h|_{\varphi(A)})^{-1}: T_1\times\cdots\times T_t\rightarrow \varphi(A)$ are left inverses of $\varphi$ and $h|_{\varphi(A)}$, respectively, then $f^{-1}=\varphi^{-1}\circ (h|_{\varphi(A)})^{-1}\circ \psi$.
\end{thm}
\begin{proof} Put $S=\{(\varphi_1(a), \ldots, \varphi_s(a)), a\in A\}$, $T=\{(\psi_1(a), \ldots, \psi_t(a)), a\in A\}$ and
$$g: S\rightarrow T,\quad g((\varphi_1(a), \ldots, \varphi_s(a)))=
(\psi_1(a), \ldots, \psi_t(a)),\,\,\, a\in A.$$ Then we have the following commutative diagram
 $$\xymatrix{
  A \ar[d]_{\varphi} \ar[r]^{f}
                & A \ar[d]^{\psi}  \\
  S  \ar[r]_{g}
                & T          }$$
Hence the results follows immediately from AGW criterion with bijections $\varphi$ and $\psi$, and the definitions of left inverses. The last assertion is obvious.
\end{proof}

As a consequence of the above result ($s=t=2$), we obtain Theorem 2 in \cite{NLQW21}.
\begin{coro}{\rm (\cite[Theorem 2]{NLQW21})} Let $A$ be a finite set, $f: A\rightarrow A$, and let $\phi=(\lambda, \eta)$ and $\bar{\phi}=( \bar{\lambda}, \bar{\eta})$ be two bijective mappings from $A$ to some subsets of $A\times A$, and denote by $\phi^{-1}, \bar{\phi}^{-1}$ their compositional inverses respectively. Let $\psi=(g, \tau): \phi(A)\rightarrow \bar{\phi}(A)$ be a mapping such that $\bar{\phi}\circ f=\psi\circ \phi$. Then $f$ is a bijective if and only if $\psi$ is bijective. Furthermore, if $\psi$ is bijective and its compositional inverse is denoted by $\psi^{-1}$, then
$$f^{-1}=\phi^{-1}\circ \psi^{-1}\circ \bar{\phi}.$$\end{coro}

We also have the following special case of Theorem \ref{Gv}.
\begin{coro}Let $ t$ be a positive integer. Let $A, S_i, T_i,  i=1, \ldots,  t$ be non-empty finite sets, and let $\varphi_i: A \rightarrow S_i$, $\psi_i: A \rightarrow T_i, 1\le i\le t$, $h: S_1\times\cdots\times S_t\rightarrow S_1\times\cdots\times S_t$ be maps such that the maps $ \varphi: A\rightarrow S_1\times\cdots\times S_t$, $\varphi(a)=(\varphi_1(a), \ldots, \varphi_t(a)), a\in A$ and $\psi: A\rightarrow S_1\times\cdots\times S_t$, $\psi(a)=(\psi_1(a), \ldots, \psi_t(a)), a\in A$ are injective and the following diagram commutes
 $$\xymatrix{
  A \ar[d]_{(\varphi_1, \ldots, \varphi_t)} \ar[r]^{f}
                & A \ar[d]^{(\psi_1, \ldots, \psi_t)}  \\
  S_1\times\cdots\times S_t  \ar[r]_{h}
                &   S_1\times\cdots\times S_t        }$$
Then $f$ is a bijection if and only if $h|_{\varphi(A)}: \varphi(A)\rightarrow S_1\times\cdots\times S_t$ is an injective. Further, if $f$ is a bijection,  $\varphi^{-1}:S_1\times\cdots\times S_t\rightarrow A$ and $(h|_{\varphi(A)})^{-1}: S_1\times\cdots\times S_t\rightarrow \varphi(A)$ are left inverses of $\varphi$ and $h|_{\varphi(A)}$, respectively, then $f^{-1}=\varphi^{-1}\circ (h|_{\varphi(A)})^{-1}\circ \psi$.\end{coro}

\begin{coro}Let $ t$ be a positive integer. Let $A, S_i,  i=1, \ldots,  t$ be non-empty finite sets, and let $\varphi_i: A \rightarrow S_i$, $\psi_i: A \rightarrow S_i, 1\le i\le t$ be surjective, $h_i: S_i\rightarrow S_i, 1\le i\le t$ be maps such that the maps $ \varphi: A\rightarrow S_1\times\cdots\times S_t$, $\varphi(a)=(\varphi_1(a), \ldots, \varphi_t(a)), a\in A$ and $\psi: A\rightarrow S_1\times\cdots\times S_t$, $\psi(a)=(\psi_1(a), \ldots, \psi_t(a)), a\in A$ are injective, $\psi_i\circ f=h\circ \varphi_i, 1\le i\le t$ and the following diagram commutes
 $$\xymatrix{
  A \ar[d]_{(\varphi_1, \ldots, \varphi_t)} \ar[r]^{f}
                & A \ar[d]^{(\psi_1, \ldots, \psi_t)}  \\
  S_1\times\cdots\times S_t  \ar[r]_{(h_1, \ldots, h_t)}
                &   S_1\times\cdots\times S_t        }$$
Then $f$ is a bijection if and only if $h_i, 1\le i\le t$ are injective. Furthermore, if $f$ is a bijection,  $\varphi^{-1}: S_1\times\cdots\times S_t\rightarrow A$ is the  left inverse of $\varphi$, and $h^{-1}_i, 1\le i\le t$ are compositional inverses of $h_i, 1\le i\le t$ respectively, then $f^{-1}=\varphi^{-1}\circ (h^{-1}_1, \ldots, h^{-1}_t )\circ \psi$.\end{coro}

In fact, Corollary 5.3 provides a possibility to solve the compositional inverses of all PPs and the process can be written as follows.

(1) Find enough possible surjections $\varphi_i, \psi_i$ and $h_i, 1\le i\le t$ such that the following diagrams commute
$$\xymatrix{
  A \ar[d]_{\varphi_i} \ar[r]^{f}
                & A \ar[d]^{\psi_i}  \\
  S_i  \ar[r]_{h_i}
                &   S_i        }$$

(2) Compute the compositional inverses $h_i^{-1}$ of $h_i, 1\le i\le t$.

(3) If we have $x=F(\varphi_1(x), \ldots, \varphi_t(x))$ for some polynomial $F(x_1, \ldots, x_t)\in A[x_1, \ldots, x_t]$, then we obtain  the compositional inverses of $f$, and we have
$$f^{-1}(x)=F\left(h_1^{-1}(\psi_1(x)), \ldots, h_t^{-1}(\psi_t(x))\right).$$

\section{An application}

As an application the results in Section 5, we will present some new PPs and their compositional inverses in this section.

Let $d>1$ be a positive integer, $q$ a prime power with $q\equiv1\pmod{d}$ and $\omega$ a  $d$-th primitive root of unity over $\mathbb{F}_q$. For $0\le i\le d-1$, let
$$A_i(x)=x^{q^{d-1}}+\omega^ix^{q^{d-2}}+\cdots+\omega^{i(d-1)}x.$$
Then we have
\begin{lem} Let the notations be as above, then we have

(i) $A_i^q(x)=\omega^iA_i(x), 0\le i\le d-1$.

(ii) For any positive integer $m$ and integers $i, j$ with $0\le i, j\le d-1$, we have
$$A_j(x)\circ A_i^m(x)=\left\{ \begin{array}{ll}
                     dA_i^m(x), & \mbox{if } j\equiv im\pmod{d}, \\
                     0, & \mbox{ otherwise.}
                    \end{array}
         \right.$$
         \end{lem}
         \begin{proof} For (i), we have
$$\left(A_i(x)\right)^q=x+\omega^ix^{q^{d-1}}+\cdots+\omega^{i(d-1)}x^q=\omega^iA_i(x).$$
This proves (i). Now we prove (ii), repeating the procedure in (i), we get
$$\left(A_i(x)\right)^{q^t}=\omega^{it}A_i(x).$$
Hence
$$\left(A_i^m(x)\right)^{q^t}=\left(A_i(x)^{q^t}\right)^m=\omega^{imt}A_i^m(x).$$
Therefore
$$ A_j(x)\circ A_i^m(x)=A_i^m(x)\sum_{t=0}^{d-1}\omega^{(jt+i(d-1-t)m)}=A_i^m(x)\omega^{im_i(d-1)}\sum_{t=0}^{d-1}\omega^{(j-im)t}$$
$$=\left\{ \begin{array}{ll}
                     d\omega^{-j}A_i^m(x), & \mbox{if } j\equiv im\pmod{d}, \\
                     0, & \mbox{ otherwise.}
                    \end{array}
         \right.$$\end{proof}

\begin{lem} Let the notations be as in Lemma, and let
$$B_i=\{A_i(x), \, x\in\mathbb{F}_{q^d}\},\,\, 0\le i\le d-1.$$
Then $B_i=\{0\}\cup y_i\mathbb{F}_q^\ast$, where $y_i$ is a fixed nonzero element of $B_i$, $0\le i\le d-1$.\end{lem}
\begin{proof} Obviously, $0\in B_i$. For any elements $y_1, y_2\in B_i$ with $y_1y_2\ne0$, by Lemma (i), we have
$$y_j^q=\omega^i y_j, \,\, j=1, 2.$$
Hence $(y_1/y_2)^q=y_1/y_2$, which implies that $y_1/y_2\in\mathbb{F}_q^\ast$, and the assertion follows. \end{proof}

\begin{thm}Let $d>1$ be a positive integer, $q$ a prime power with $q\equiv1\pmod{d}$ and $\omega$ a  $d$-th primitive root of unity over $\mathbb{F}_q$. Let
$$A_i(x)=x^{q^{d-1}}+\omega^ix^{q^{d-2}}+\cdots+\omega^{i(d-1)}x, \,\, 0\le i\le d-1.$$ Let $m_0, \ldots, m_{d-1}$ be positive integers  and $u_0, \ldots, u_{d-1}\in\mathbb{F}_q$. Then the polynomial
$$f(x)=\sum_{t=0}^{d-1}u_iA_i^{m_i}(x),$$
 is a PP over $\mathbb{F}_{q^d}$ if and only if $\{im_i, 0\le i\le d-1\}$ is a complete residue modulo $d$, $u_0\cdot\cdots\cdot u_{d-1}\in\mathbb{F}_q^\ast$ and $\gcd(m_0\cdot\cdots\cdot m_{d-1}, q-1)=1$. Furthermore, if $f(x)$ is a PP over $\mathbb{F}_{q^d}$, and $r_i$ are positive integers with $m_ir_i\equiv1\pmod{d(q-1)}$, then
$$f^{-1}(x)=\frac{1}{d}\sum_{i=0, j\equiv im_i\pmod{d}}^{d-1}\omega^i(du_i\omega^{-j})^{-r_i}A_j(x)^{r_i}.$$\end{thm}

\begin{proof} We first prove the necessity. Suppose that there is a nonnegative integer $j$ such that $j\not\equiv im_i\pmod{d}$ for any $i\in\{0, \ldots, d-1\}$. By Lemma (ii), we have
$$A_j(x)\circ\left(u_iA_i^{m_i}(x)\right)=0, \quad 0\le i\le d-1,$$
hence $ A_j(x)\circ f(x)=0$ and $f(x)$ is not a PP over $\mathbb{F}_{q^d}$. On the other hand, if $u_i=0$ for some $i, 0\le i\le d-1$, then it is trivial that $\{tm_t, 0\le t\le d-1, t\ne i\}$ is not a complete residue modulo $d$, so $f(x)$ is not a PP over $\mathbb{F}_{q^d}$.

Now we assume that $u_0\cdot\cdots\cdot u_{d-1}\in\mathbb{F}_q^\ast$ and $\{im_i, 0\le i\le d-1\}$ is a complete residue modulo $d$. Then
$$A_j(x)\circ f(x)=du_i\omega^{-j}A_i^{m_i}(x),\,\, im_i\equiv j\pmod{d}, \, 0\le j\le d-1,$$ so we have the following commutative diagram
$$\xymatrix{
  \mathbb{F}_{q^d} \ar[d]_{A_i(x)} \ar[r]^{f}
                & \mathbb{F}_{q^d} \ar[d]^{A_j(x)}  \\
  B_i  \ar[r]_{du_i\omega^{-j}x^{m_i}}
                & B_j          } $$
for each $j, 0\le j\le d-1$, where $i$ is the integer with $im_i\equiv j\pmod{d}$, $B_i=\{A_i(x), x\in \mathbb{F}_{q^d}\}, 0\le i\le d-1$. By Lemma 6.2, the map
$$h_{ij}: B_i\to B_j, \,\, a\mapsto du_i\omega^{-j}a^{m_i},$$
is bijective if and only if $\gcd(m_i, q-1)=1$. If $\gcd(m_i, q-1)=1$, then the compositional inverse of $du_i\omega^{-j}x^{m_i}$ is $(du_i\omega^{-j})^{-r_i}x^{r_i}$, where $m_ir_i\equiv1\pmod{d(q-1)}$. Note that
$$x=\frac{1}{d}\sum_{i=0}^{d-1}\omega^i A_i(x),$$
by Corollary 5.3, we have
$$f^{-1}(x)=\frac{1}{d}\sum_{i=0, j\equiv im_i\pmod{d}}^{d-1}\omega^i(du_i\omega^{-j})^{-r_i}A_j(x)^{r_i}.$$

\end{proof}

Applied the above theorem to $m_i=q^d-2, 0\le i\le d-1$, then $r_i=m_i=q^d-2, 0\le i\le d-1$, and we have the following corollary.

\begin{coro}Let $d>1$ be a positive integer, $q$ a prime power with $q\equiv1\pmod{d}$ and $\omega$ a  $d$-th primitive root of unity over $\mathbb{F}_q$. Let
$$A_i(x)=x^{q^{d-1}}+\omega^ix^{q^{d-2}}+\cdots+\omega^{i(d-1)}x, \,\, 0\le i\le d-1.$$ For any $u_0, \ldots, u_{d-1}\in\mathbb{F}_q^\ast$ with $u_i\omega^i=u_j\omega^j, 0\le i, j\le d-1$,   the polynomials
$$f(x)=\sum_{t=0}^{d-1}u_iA_i^{q^d-2}(x),$$
 are  involutions over $\mathbb{F}_{q^d}$. \end{coro}

\begin{thm} Let $d>1$ be a positive integer, $q$ a prime power with $\gcd(q, d)=1$.  For $u_1, u_2\in\mathbb{F}_q^\ast$ and a positive integer $m$, the polynomial
$$f(x)=u_1(x^q-x)+u_2\left(x+x^q+\cdots+x^{q^{d-1}}\right)^{m}$$
is a PP over $\mathbb{F}_{q^d}$ if and only if $\gcd(m, q-1)=1$. Moreover, if $f(x)$ is a PP over $\mathbb{F}_{q^d}$, then
$$f^{-1}(x)=\frac{1}{d}(u_2d)^{-r}\left(x+x^q+\cdots+x^{q^{d-1}}\right)^r+\frac{d-1}{2du_1}\left(x+x^q+\cdots+x^{q^{d-1}}\right)-\frac{1}{du_1}\sum_{i=1}^{d-1}ix^{q^{d-1-i}},$$
where $r$ is a positive integer with $mr\equiv1\pmod{q-1}$.\end{thm}
\begin{proof}Note that  $x+x^q+\cdots+x^{q^{d-1}}=Tr(x)$. By the assumptions, we have the following diagrams
$$\xymatrix{
  \mathbb{F}_{q^d} \ar[d]_{Tr(x)} \ar[r]^{f}
                & \mathbb{F}_{q^d} \ar[d]^{Tr(x)}  \\
  \mathbb{F}_q \ar[r]_{du_2x^{m}}
                & \mathbb{F}_q          }  \quad
                \xymatrix{
  \mathbb{F}_{q^d} \ar[d]_{x^q-x} \ar[r]^{f}
                & \mathbb{F}_{q^d} \ar[d]^{x^q-x}  \\
  B \ar[r]_{u_1(x^q-x)}
                & B          }$$
  where $B=\{x^q-x, x\in\mathbb{F}_{q^d}\}$. It is easy to see that the map $h: \mathbb{F}_q\to\mathbb{F}_q, a\mapsto du_2a^m$ is bijective if and only if $\gcd(m, q-1)=1$. On the other hand, we have
  $$dx=x+x^q+\cdots+x^{q^{d-1}}-\sum_{j=1}^{d-1}j\left(x^q-x\right)^{q^{d-1-j}},$$
  so $f(x)$ is injective on each $Tr^{-1}(s)$, $s\in\mathbb{F}_q$. Therefore the polynomial $f(x)$ is a PP over $\mathbb{F}_{q^d}$ if and only if $\gcd(m, q-1)=1$ by AGW criterion.

  Since
  $$\left(x^{q^{d-2}}+2x^{q^{d-3}}+\cdots+(d-1)x\right)\circ(x^q-x)=Tr(x)-dx,$$
  so  the compositional inverse of the function $u_1(x^q-x)$ from $B=\{x^q-x, x\in\mathbb{F}_{q^d}\}$ is
  $$\frac{-1}{du_1}\left(x^{q^{d-2}}+2x^{q^{d-3}}+\cdots+(d-1)x\right).$$
  Note that the compositional inverse of the function $du_2x^m$ from $\mathbb{F}_q$ to $\mathbb{F}_q$ is given by $(du_2)^{-r}x^r$, by Corollary 5.3, we have
  $$f^{-1}(x)=\frac{1}{d}(u_2d)^{-r}\left(x+x^q+\cdots+x^{q^{d-1}}\right)^r+\frac{d-1}{2du_1}\left(x+x^q+\cdots+x^{q^{d-1}}\right)-\frac{1}{du_1}\sum_{i=1}^{d-1}ix^{q^{d-1-i}}.$$

  This completes the proof.
                \end{proof}

\section{Concluding remarks}
If it is difficult to find  enough  surjections $\varphi_i, \psi_i$ and bijections $h_i, 1\le i\le t$ such that the following diagrams commute
$$\xymatrix{
  A \ar[d]_{\varphi_i} \ar[r]^{f}
                & A \ar[d]^{\psi_i}  \\
  S_i  \ar[r]_{h_i}
                &   S_i        }$$
we can use the dual diagrams of the above diagrams 
$$\xymatrix{
  A \ar[d]_{\psi} \ar[r]^{f^{-1}}
                & A \ar[d]^{\varphi}  \\
  S_i  \ar[r]_{h_i^{-1}}
                &   S_i        }$$
as we did in \cite{Y22}. In \cite{Y22}, we obtained the compositional inverses of some PPs by using one diagram and its dual diagram. Now we can make use of all diagrams and their dual diagrams to find the compositional inverses of some specified PPs. However, now we have not had such examples.

\end{document}